\documentclass[12pt]{article}
\usepackage{epsfig,amsmath,amssymb,latexsym,url}

\newcommand{\RR}{\mathbb{R}}

\def\bfa{\mathbf{a}}
\def\bfb{\mathbf{b}}

\def\bft{\mathbf{t}}
\def\bfu{\mathbf{u}}
\def\bfv{\mathbf{v}}

\def\bfx{\mathbf{x}}
\def\bfy{\mathbf{y}}

\def\tbar{\overline{t}}
\def\ubar{\overline{u}}
\def\vbar{\overline{v}}

\def\bflambda{{\boldsymbol \lambda}}

\def\per{{\rm per}}
\def\tr{{\rm tr}}

\newtheorem{theorem}{Theorem}
\newtheorem{lemma}{Lemma}

\newtheorem{corollary}{Corollary}

\newenvironment{proof}{\begin{trivlist}\item[]{\emph{Proof.}}}
               {\hfill$\Box$\end{trivlist}}

\def\Acknowledgement{\goodbreak\bigskip\noindent{\bf Acknowledgement.\ }}

\begin{document}

\title{Bounds on the determinant \\ of an exponential matrix}
\author{
Michael S. Floater\footnote{
Department of Mathematics,
University of Oslo,
PO Box 1053, Blindern,
0316 Oslo,
Norway,
{\it email: michaelf@math.uio.no}}
}                                                                               
                                                                                
\maketitle

\begin{abstract}
We derive upper and lower bounds on the determinant of an exponential matrix.
They can be transformed into corresponding bounds
for the determinant of a univariate Gaussian matrix.
\end{abstract}


\smallskip

\noindent {\em Math Subject Classification: }
15A15, 
33B10, 
65D05, 
65F40. 

\section{Introduction}
Consider the $n \times n$ matrix
\begin{equation}\label{eq:A}
A := [e^{x_iy_j}]_{i,j=1,\ldots,n} =
\left[
  \begin{matrix}
  e^{x_1 y_1} & \cdots & e^{x_1 y_n} \\
 \vdots & & \vdots \\
  e^{x_n y_1} & \cdots & e^{x_n y_n}
  \end{matrix}
  \right],
\end{equation}
for real $x_1 < x_2 < \cdots < x_n$ and $y_1 < y_2 < \cdots < y_n$.
It includes the Vandermonde and generalized Vandermonde matrices
as special cases, see, for example,
Pinkus~\cite[Section 4.2]{Pinkus:2009}.
In this paper, we study upper and lower bounds on the
determinant of $A$. The motivation comes
from Schumaker's conjecture
about partial interpolation by multivariate
Bernstein polynomials~\cite{Alfeld:2005}, \cite{Jaklic:2014},
\cite{Floater:2023}. The conjecture can be recast as the
non-singularity of Hadamard (element-wise) products of matrices in the
form of $A$, and the hope is that bounds on the determinant of $A$,
or the technique used to find them,
may help in the future towards a solution to the conjecture.
Bounds on the determinant $A$ have also been applied, either directly
or indirectly, to the error analysis of interpolation by exponential
functions by Yarotsky~\cite{Yarotsky:2013} and to finding optimal points
for interpolation by exponentials and
Gaussian radial basis functions by Bos and de Marchi~\cite{BosdeMarchi:2011}.

Let us recall some basic facts about $A$.
The determinant of~$A$ is positive, which was shown by
P\'{o}lya in 1920 using a form of Descartes' Rule of Signs,
see P\'{o}lya and Szeg\H{o}~\cite[Part Five, Problem 76]{Polya:1976}
and Gantmacher and Krein~\cite{Gantmakher:1937}.
Since all square submatrices of~$A$ have the same form as~$A$
itself, they too have positive determinant, and thus~$A$
is totally positive in the sense that all its minors are positive.
Every totally positive matrix satisfies the Hadamard inequality:
its determinant is bounded above by the product of its diagonal elements
(this follows from Sylvester's determinant identity;
see~\cite[Th{\'e}or{\`e}me 1]{Gantmakher:1937},
\cite[Section 1.4]{Pinkus:2009},
and~\cite[Chapter 6]{FallatJohnson:2011}), and so
$$ \det(A) \le e^{x_1 y_1 + \cdots + x_n y_n}. $$
Another consequence of total positivity is the
spectral theory of Gantmacher and Krein~\cite{Gantmakher:1937}
(which builds on Perron-Frobenius theory and properties of compound matrices).
So the eigenvalues of~$A$ are real, positive, and distinct, and
the eigenvectors satisfy sign-change properties;
see~\cite[Theorem 5.3]{Pinkus:2009}.

We can also express the determinant of $A$
using the remarkable Harish-Chandra-Itzykson-Zuber
formula~\cite{HarishChandra:1957,ItzyksonZuber:1980,GrossRichards:1989},
which involves an integral over unitary matrices (see Section~2),
and is connected to
Schur functions~\cite{Macdonald:1995,deMarchi:2001,Ait-Haddou:2019}.
While this formula is far from elementary to derive,
it was observed by Bos and de Marchi~\cite{BosdeMarchi:2011}
that the integrand can be bounded above and below using
the trace inequality of Coope and Renaud~\cite{CoopeRenaud:2009},
and this leads to corresponding bounds on $\det(A)$,
which are as follows.
Let $\bfx := (x_1,x_2,\ldots,x_n)$ and
$\bfy := (y_1,y_2,\ldots,y_n)$ and denote by $V(\bfx)$ the
Vandermonde determinant
$$ V(\bfx) = V(x_1,x_2,\ldots,x_n)
:= \prod_{1 \le i < j \le n} (x_j-x_i) $$
(the determinant of the matrix $[x_i^{j-1}]_{i,j=1,\ldots,n}$).
Let $c_n$ be the `superfactorial'
$$ c_n := \prod_{i=1}^n i!. $$
Then the bounds are
\begin{theorem}\label{thm:weak}
\begin{align*}
e^{x_1 y_n + x_2 y_{n-1} + \cdots + x_n y_1}
  \le \frac{c_{n-1} \det(A)}{V(\bfx) V(\bfy)} \le
   e^{x_1 y_1 + x_2 y_2 + \cdots + x_n y_n}.
\end{align*}
\end{theorem}
On the other hand, it would
be preferable to have a more elementary proof of Theorem~\ref{thm:weak}.
This turns out to be the case, due to the fact that
we can alternatively express $\det(A)$ recursively, as an integral over
a rectangle in $\RR^{n-1}$ in which the integrand is
itself an exponential determinant, but of order $n-1$
(Lemma~\ref{lem:reductionupper}).
In this way we obtain Theorem~\ref{thm:weak}
by proving it first in the case $n=2$, and using induction on $n$ for
$n \ge 3$.
But not only that, using this recursive approach, we obtain
better bounds, which are as follows.
Define the sum,
$$ s(\bfx) := x_1 + x_2 + \cdots + x_n, $$
and denote the permanent of $A$ by $\per(A)$.
\begin{theorem}\label{thm:main}
$$ \frac{1}{c_{n-1}} e^{s(\bfx)s(\bfy)/n}
 \le \frac{\det(A)}{V(\bfx) V(\bfy)} \le
   \frac{1}{c_n} \per(A). $$
\end{theorem}
To see that this is an improvement,
let $S_n$ be the symmetric group of permutations of
$(1,2,\ldots,n)$, and for each permutation $\sigma \in S_n$, let
\begin{equation}\label{eq:phi}
 \phi_\sigma := \sum_{i=1}^n x_i y_{\sigma(i)}.
\end{equation}
Then the rearrangement inequality~\cite[Thm 368]{Hardy:1952} bounds
$\phi_\sigma$:
\begin{equation}\label{eq:rearrangement}
\sum_{i=1}^n x_i y_{n+1-i}
   \le \phi_\sigma \le \sum_{i=1}^n x_i y_{i}.
\end{equation}
Let $m_a$ and $m_g$ be the arithmetic and geometric means
of the exponentials $e^{\phi_\sigma}$, $\sigma \in S_n$,
$$
m_a := \frac{1}{n!} \sum_{\sigma \in S_n} e^{\phi_\sigma}, \quad
m_g :=
 \Big(\prod_{\sigma \in S_n} e^{\phi_\sigma} \Big)^{1/n!}. $$
Then,
$$ m_a = \frac{1}{n!} \per(A),
$$
and since
$$ \sum_{\sigma\in S_n} \phi_\sigma
= \sum_{i=1}^n x_i \sum_{\sigma\in S_n} y_{\sigma(i)}
= (n-1)! s(\bfx) s(\bfy), $$
we find
$$ m_g = e^{s(\bfx)s(\bfy)/n}, $$
and so Theorem~\ref{thm:main} can be rewritten as
$$ m_g \le \frac{c_{n-1} \det(A)}{V(\bfx)V(\bfy)} \le m_a. $$
and by (\ref{eq:rearrangement}),
$$ e^{x_1 y_n + x_2 y_{n-1} + \cdots + x_n y_1}
  < m_g < m_a < e^{x_1 y_1 + x_2 y_2 + \cdots + x_n y_n}.
$$

While the recursive proofs of Theorems~\ref{thm:weak} and \ref{thm:main}
are elementary, they are not simple either.
When we apply the induction step, 
the difficult task is to compute the resulting integral,
and the integrand is either a Vandermonde determinant or
the product of a Vandermonde determinant and a function of the
sum of the variables.
It turns out that there is an identity relating this kind of integral
to divided differences, which is derived in~\cite{Floater:2026a}.
For the lower bound of Theorem~\ref{thm:main},
we also need an estimate on
a divided difference from below, and for this we use an error
formula from~\cite{Floater:2003}.
For a survey of divided differences, see~de Boor~\cite{deBoor:2005},
and for some basic theory, see
Isaacson and Keller~\cite[Section 6.1]{Isaacson:1994}
and Steffensen~\cite{Steffensen:1927}.
For the upper bound of Theorem~\ref{thm:main},
we need an auxiliary result concerning the
maximization of a function that is both convex and symmetric.

The bounds of Theorem~\ref{thm:main}
can be transformed into corresponding bounds
for the determinant of a univariate Gaussian matrix,
and an application of the lower bound
is a simple method for selecting the shape parameter.
We treat this in Section~\ref{sec:gaussian}.

\section{The unitary matrix integral}

The unitary matrix integral formula
((3.15) of Gross and Richards~\cite{GrossRichards:1989})
is
\begin{equation}\label{eq:unitary}
 \frac{c_{n-1} \det(A)}{V(\bfx) V(\bfy)}
   = \int_{U(n)} e^{\tr(XUYU^*)} \, dU,
\end{equation}
where the integration is over the group of $n \times n$ unitary
matrices $U$ with respect to the normalized Haar measure.
Here,
$X = {\rm diag}(x_1,x_2,\ldots,x_n)$,
$Y = {\rm diag}(y_1,y_2,\ldots,y_n)$,
$U^*$ is the conjugate transpose of $U$, and ${\tr}$ denotes the
trace.
Multiplying out the matrices in the integrand, and taking
the trace gives
$$ \tr(XUYU^*) = \bfx^T (U \circ \overline{U}) \bfy, $$
with $\bfx$ and $\bfy$ treated as column vectors, and where
$\circ$ denotes the Hadamard product.
It follows that if $U = [u_{ij}]$, then
$$ \tr(XUYU^*) = \bfx^T V \bfy, $$
where $V = [v_{ij}]$ is the real matrix with elements
$v_{ij} = |u_{ij}|^2$.
Since $U$ is unitary, $V$ is doubly stochastic:
$v_{ij} \ge 0$, and
$$ \sum_{j=1}^n v_{ij} = 1, \quad i=1,\ldots,n; \qquad
 \sum_{i=1}^n v_{ij} = 1, \quad j=1,\ldots,n. $$
So by the Birkhoff-von Neumann theorem~\cite[p. 63]{Bapat:1997},
we can express $V$ as a convex combination of the permutation matrices
$P_\sigma = [p_{ij}]$ defined by
$$ p_{ij} = \begin{cases} 1 & \hbox{if $j = \sigma(i)$;} \cr
                          0 & \hbox{otherwise}.
\end{cases}
$$
So there are weights $\lambda_\sigma \ge 0$ with
$\sum_{\sigma \in S_n} \lambda_\sigma = 1$ such that
$$ V = \sum_{\sigma \in S_n} \lambda_\sigma P_\sigma. $$
Therefore,
$$ \bfx^T V \bfy = 
\sum_{\sigma \in S_n} \lambda_\sigma \bfx^T P_\sigma \bfy
 = \sum_{\sigma \in S_n} \lambda_\sigma \phi_\sigma, $$
with $\phi_\sigma$ as in (\ref{eq:phi}),
and so by (\ref{eq:rearrangement}),
$$ \sum_{i=1}^n x_i y_{n+1-i}
   \le \bfx^T V \bfy \le \sum_{i=1}^n x_i y_i. $$
Applying this to the integrand of (\ref{eq:unitary}),
and since $dU$ is normalized to give volume $1$, we
deduce Theorem~\ref{thm:weak}.

\section{The case $n=2$}

Next we give an elementary proof of both
Theorems~\ref{thm:weak} and~\ref{thm:main} when $n=2$. In this case,
$$ \det(A) = e^b - e^a, $$
where $a := x_1 y_2 + x_2 y_1$ and $b := x_1 y_1 + x_2 y_2$,
and
$$ V(\bfx)V(\bfy) = (x_2-x_1)(y_2-y_1) = b-a, $$
and so
$$ \frac{\det(A)}{V(\bfx)V(\bfy)}
   = \frac{e^b - e^a}{b-a}, $$
a first order divided difference of the exponential function.
By the mean value theorem, there is some $\xi \in (a,b)$
such that
$$ \frac{e^b - e^a}{b-a} = e^\xi, $$
and so
$$ e^a \le \frac{e^b - e^a}{b-a} \le e^b, $$
and
$$ e^{x_1y_2+x_2y_1} \le \frac{\det(A)}{V(\bfx)V(\bfy)}
 \le e^{x_1y_1+x_2y_2}, $$
which is Theorem~\ref{thm:weak} when $n=2$.
However, we can improve on these bounds by expressing
the difference $e^b-e^a$ in terms of the hyperbolic sine,
$$ e^b - e^a = 2 e^{(a+b)/2} \sinh\big((b-a)/2\big). $$
Using the well known estimates
$$ t \le \sinh(t) \le t \cosh(t), \qquad t \ge 0, $$
we then find that
$$ e^{(a+b)/2} \le \frac{e^b - e^a}{b-a} \le \frac{e^a + e^b}{2}, $$
and so
$$ e^{(x_1+x_2)(y_1+y_2)/2} \le
   \frac{\det(A)}{V(\bfx)V(\bfy)} \le \frac{1}{2} \per(A), $$
which is Theorem~\ref{thm:main} when $n=2$.

\section{A multiple integral}
To prove~Theorems~\ref{thm:weak} and~\ref{thm:main}
for $n \ge 3$ we express $\det(A)$ as a multiple integral
whose integrand is the determinant of an exponential
matrix of order $n-1$.
Let $R(\bfx) \subset \RR^{n-1}$ denote the `sequential' rectangle
\begin{equation}\label{eq:sequential}
 R(\bfx) := 
[x_1,x_2] \times [x_2,x_3] \times \cdots \times [x_{n-1},x_n],
\end{equation}
and for a continuous function $f:R(\bfx) \to \RR$, denote
its integral in $R(\bfx)$ as
$$ \int_{R(\bfx)} f(\bft) \, d \bft
 := 
\int_{x_{n-1}}^{x_n} \cdots
\int_{x_2}^{x_3}
\int_{x_1}^{x_2} f(t_1,t_2,\ldots,t_{n-1}) \, dt_1 \, dt_2 \cdots dt_{n-1},
$$
where $\bft := (t_1,t_2,\ldots,t_{n-1}) \in R(\bfx)$.
Consider again $A$ in (\ref{eq:A}) with $n \ge 3$.
Choose any column index $k \in \{1,2,\ldots,n\}$ and define
\begin{equation}\label{eq:uj}
u_j := \begin{cases} y_j-y_k, & j=1,\ldots,k-1; \cr
                       y_{j+1}-y_k & j=k,\ldots,n-1, \cr
\end{cases}
\end{equation}
and
\begin{equation}\label{eq:pk}
 \omega_k := (-1)^{k-1} u_1 \cdots u_{n-1}.
\end{equation}
Then
$$ u_1 < \cdots < u_{k-1} < 0 < u_k < \cdots < u_{n-1}, 
\quad\hbox{and} \quad \omega_k > 0. $$

\begin{lemma}\label{lem:reductionupper}
$$
\det(A) = e^{s(\bfx)y_k} \omega_k
 \int_{R(\bfx)} \det(B_k(\bft)) \, d \bft,
$$
where
$$
B_k(\bft) := \left[
\begin{matrix}
e^{t_1 u_1} & \cdots & e^{t_1 u_{n-1}} \\
\vdots & & \vdots \\
e^{t_{n-1} u_1} & \cdots & e^{t_{n-1} u_{n-1}}
\end{matrix}
\right].
$$
\end{lemma}

\begin{proof}
For $i=1,\ldots,n$, we divide the $i$-th row of $A$ by $e^{x_i y_k}$
so that
$$
\det(A) = e^{x_1y_k} \cdots e^{x_ny_k} \det(A')
	= e^{s(\bfx)y_k} \det(A'),
$$
where
$$
A' :=
\left[
\begin{matrix}
e^{x_1 u_1} & \cdots & e^{x_1 u_{k-1}} & 1 &
e^{x_1 u_k} & \cdots & e^{x_1 u_{n-1}} \\
\vdots & & \vdots & \vdots & \vdots & & \vdots \\
e^{x_n u_1} & \cdots & e^{x_n u_{k-1}} & 1 &
e^{x_n u_k} & \cdots & e^{x_n u_{n-1}}
\end{matrix}
\right].
$$
Next we subtract row $n-1$ of $A'$ from row $n$ which does not
change $\det(A')$,
and write the differences as
$$ e^{x_n u_j} - e^{x_{n-1} u_j}
  = u_j \int_{x_{n-1}}^{x_n} e^{tu_j} \, dt, 
 \quad j=1,\ldots,n-1. $$
Then, with $\int = \int_{x_{n-1}}^{x_n} dt$, we can
replace row $n$ by
\begin{equation}\label{eq:newrow}
 \Big[u_1 \int e^{tu_1},\ldots,u_{k-1} \int e^{tu_{k-1}},0, 
  u_k \int e^{tu_k},\ldots,u_{n-1}\int e^{tu_{n-1}}\Big].
\end{equation}
Next we subtract row $n-2$ from row $n-1$,
and then we can replace row $n-1$ by (\ref{eq:newrow}) but now
with $\int = \int_{x_{n-2}}^{x_{n-1}} dt$.
We continue like this until we have modified the second row of $A'$ to
(\ref{eq:newrow}) with $\int = \int_{x_1}^{x_2} dt$.
Now $\det(A')$ reduces to a
determinant of order $n-1$, which can be written as the
product of $\omega_k$ and the stated multiple integral
of $\det(B_k(\bft))$.
\end{proof}

\section{The bounds of Theorem~\ref{thm:weak}}

We now prove Theorem~\ref{thm:weak} in general and
suppose that $n \ge 3$.
We can choose any $k \in \{1,\ldots,n\}$ and
apply Lemma~\ref{lem:reductionupper}, and since
$t_1 < t_2 < \cdots < t_{n-1}$ and
$u_1 < u_2 < \cdots < u_{n-1}$ in the matrix $B_k(\bft)$,
we can apply the induction hypothesis to deduce that
$$ e^{t_1u_{n-1} + \cdots + t_{n-1}u_1} \le
\frac{c_{n-2}\det(B_k(\bft))}{V(\bft)V(\bfu)}
 \le e^{t_1u_1 + \cdots + t_{n-1}u_{n-1}}. $$
Since $V(\bfu) = V(y_1,\ldots,y_{k-1},y_{k+1},\ldots,y_n)$ and
$$ \omega_k = (y_k-y_1)\cdots(y_k-y_{k-1})
         (y_{k+1}-y_k)\cdots(y_n-y_k), $$
we find
\begin{equation}\label{eq:pkVu}
\omega_k V(\bfu) = V(\bfy),
\end{equation}
and so
$$ \int_{R(\bfx)} V(\bft) e^{t_1u_{n-1} + \cdots + t_{n-1}u_1} \, d\bft
\le \frac{c_{n-2}\det(A)}{V(\bfy)e^{s(\bfx)y_k}} \le
\int_{R(\bfx)} V(\bft) e^{t_1u_1 + \cdots + t_{n-1}u_{n-1}} \, d\bft. $$
Since $u_1,\ldots,u_{k-1}$ are negative
and $u_k,\ldots,u_{n-1}$ are positive,
\begin{align*}
 t_1u_{n-1} + \cdots + t_{n-k}u_k &
\ge x_1u_{n-1} + \cdots + x_{n-k}u_k, \cr
 t_{n+1-k}u_{k-1} + \cdots + t_{n-1}u_1 &
\ge x_{n+2-k}u_{k-1} + \cdots + x_nu_1,
\end{align*}
and so
$$ s(\bfx)y_k + t_1u_{n-1} + \cdots + t_{n-1}u_1
\ge x_1y_n + \cdots + x_ny_1. $$
Similarly,
\begin{align*}
t_1u_1 + \cdots + t_{k-1}u_{k-1}
& \le x_1u_1 + \cdots + x_{k-1}u_{k-1}, \cr
t_ku_k + \cdots + t_{n-1}u_{n-1}
& \le x_{k+1}u_k + \cdots + x_nu_{n-1},
\end{align*}
and so
$$
s(\bfx)y_k + t_1u_1 + \cdots + t_{n-1}u_{n-1}
 \le x_1y_1 + \cdots + x_ny_n. $$
Therefore,
$$ e^{x_1y_n + \cdots + x_ny_1}
\int_{R(\bfx)} V(\bft) \, d\bft
\le \frac{c_{n-2}\det(A)}{V(\bfy)} \le
e^{x_1y_1 + \cdots + x_ny_n}
\int_{R(\bfx)} V(\bft) \, d\bft. $$
By Corollary~1 of~\cite{Floater:2026a},
\begin{equation}\label{eq:floater}
 \int_{R(\bfx)} V(\bft) \, d\bft = \frac{1}{(n-1)!} V(\bfx),
\end{equation}
and observing that $c_{n-2} (n-1)! = c_{n-1}$ completes the proof.

\section{Divided difference error formula}

For proving the lower bound of Theorem~\ref{thm:main}, we
use a formula for the error incurred when approximating
a divided difference by a normalized derivative.
Let $v_1,v_2,\ldots,v_n \in \RR$, let
$\vbar$ be their mean
$$ \vbar := \frac{v_1 + v_2 + \cdots + v_n}{n}, $$
and let
$$ S = \sum_{i=0}^n (v_i-\vbar)^2. $$

\begin{lemma}\label{lem:ddid}
If $f$ is $C^{n+1}$ in the smallest interval containing
$v_1,\ldots,v_n$ then there exists $\xi$ in that interval such that
$$ [v_1,v_2,\ldots,v_n]f = \frac{f^{(n-1)}(\vbar)}{(n-1)!}
+ \frac{S}{2} \frac{f^{(n+1)}(\xi)}{(n+1)!}. $$
\end{lemma}

\begin{proof}
This was shown in Theorem~2 and Lemma 3 of~\cite{Floater:2003}.
\end{proof}

\section{The lower bound of Theorem~\ref{thm:main}}

To prove the lower bound of Theorem~\ref{thm:main},
we suppose $n \ge 3$.
We can apply Lemma~\ref{lem:reductionupper}
with any $k \in \{1,\ldots,n\}$,
and by the induction hypothesis,
$$
 \det(B_k(\bft)) \ge 
\frac{1}{c_{n-2}}
V(\bft) V(\bfu) e^{s(\bft)s(\bfu)/(n-1)}.
$$
and recalling (\ref{eq:pkVu}),
\begin{equation}\label{eq:low1}
 \det(A) \ge \frac{1}{c_{n-2}} V(\bfy) e^{s(\bfx)y_k}
\int_{R(\bfx)} V(\bft) e^{s(\bft)s(\bfu)/(n-1)} \, d\bft.
\end{equation}
If $s(\bfu) = 0$, then, by (\ref{eq:floater}),
the integral in (\ref{eq:low1}) becomes $V(\bfx)/(n-1)!$,
and the lower bound follows from the fact that $y_k = s(\bfy)/n$.
Otherwise, $s(\bfu) \ne 0$,
and we can write the integral in (\ref{eq:low1}) as
$$ \int_{R(\bfx)} V(\bft) f^{(n-1)}(s(\bft)) \, d\bft, $$
where
$$ f(\alpha) := \frac{e^{\ubar \alpha}}{\ubar^{n-1}}
\quad \hbox{and} \quad
\ubar := \frac{s(\bfu)}{n-1} \ne 0. $$
Then by Theorem~1 of~\cite{Floater:2026a},
$$ \int_{R(\bfx)} V(\bft) f^{(n-1)}(s(\bft)) \, d\bft
= V(\bfx) [p_1,p_2,\ldots,p_n]f, $$
where
\begin{equation}\label{eq:pi}
 p_i := x_1 + \cdots + x_{n-i} + x_{n+2-i} + \cdots + x_n
      = \Big(\sum_{j=1}^n x_j \Big) - x_{n+1-i}.
\end{equation}
Since $f^{(n+1)}(\alpha) = \ubar^2 e^{\ubar \alpha} \ge 0$ for
all $\alpha \in \RR$,
Lemma~\ref{lem:ddid} implies
$$ [p_1,p_2,\ldots,p_n]f \ge
\frac{1}{(n-1)!}
f^{(n-1)}\Big(\frac{p_1 + \cdots + p_n}{n}\Big)
= \frac{1}{(n-1)!}
e^{\ubar(p_1 + \cdots + p_n)/n}. $$
From (\ref{eq:pi}),
$$ p_1 + \cdots + p_n = 
      n \Big(\sum_{j=1}^n x_j \Big) - \sum_{i=1}^n x_{n+1-i}
       = (n-1) \sum_{j=1}^n x_j, $$
and so
$$ \int_{R(\bfx)} V(\bft) f^{(n-1)}(s(\bft)) \, d\bft
 \ge \frac{1}{(n-1)!} V(\bfx) e^{s(\bfu)s(\bfx)/n}, $$
and
$$ \det(A) \ge \frac{1}{c_{n-1}} 
V(\bfx) V(\bfy) e^{s(\bfx)(y_k+s(\bfu)/n)}. $$ 
Since
$$ y_k+\frac{s(\bfu)}{n}
 = \frac{1}{n} \Big(ny_k + \sum_{j=1}^n (y_j-y_k)\Big)
   = \frac{s(\bfy)}{n}, $$
the proof is complete.

\section{Maximization on $R(\bfx)$}

In order to prove the upper bound of 
Theorem~\ref{thm:main}, we will make use of the following
auxiliary result. Observe that the sequential rectangle $R(\bfx)$
in (\ref{eq:sequential}) is contained in the $(n-1)$-dimensional
cube $[x_1,x_n]^{n-1}$.
\begin{lemma}\label{lem:axial2}
For a convex function $f:[x_1,x_n]^{n-1} \to \RR$
that is symmetric in its variables,
$$ \max_{\bft \in R(\bfx)} f(\bft)
= \max_{i=1,\ldots,n} f(\bfv_i), $$
where
\begin{equation}\label{eq:vi}
\bfv_i := (x_1,\ldots,x_{i-1},x_{i+1},\ldots,x_n), \quad i=1,\ldots,n.
\end{equation}
\end{lemma}
So $f$ attains its maximum in $R(\bfx)$ at one of $n$
of the $2^{n-1}$ vertices of $R(\bfx)$.

\begin{proof}
Since $f$ is convex on $R(\bfx)$, it attains its maximum at a vertex
of $R(\bfx)$.
Let $\bft = (t_1,\ldots,t_{n-1})$ be a vertex of $R(\bfx)$. Then
$$ t_j \in \{x_j,x_{j+1}\}, \quad j=1,\ldots,n-1. $$
So $x_1$ and $x_n$ occur at most once in $\bft$
but for $2 \le j \le n-1$, $x_j$
may occur twice, in which case
\begin{equation}\label{eq:repeat}
t_{j-1} = x_j = t_j.
\end{equation}
Let $m(\bft)$ denote the number of repeated $x_j$ in $\bft$,
i.e., the number of indices $j \in \{2,\ldots,n-1\}$
such that (\ref{eq:repeat}) holds.
We will show that if $m(\bft) > 0$, then there are vertices
$\bft'$ and $\bft''$ of $R(\bfx)$
with $m(\bft') = m(\bft'') = m(\bft) - 1$
such that
\begin{equation}\label{eq:reduce}
f(\bft) \le \max\{f(\bft'),f(\bft'')\}.
\end{equation}
By repeating this argument, this will imply
that the maximum of $f$ over the vertices of $R(\bfx)$
is attained at a vertex $\bft$ with $m(\bft) = 0$,
in which case $\bft$ is one of the $\bfv_i$ in (\ref{eq:vi}).

To prove (\ref{eq:reduce}),
choose some $j \in \{2,\ldots,n-1\}$ such that
(\ref{eq:repeat}) holds.
Then there must be at least one index
$i \in \{1,\ldots,j-1\}$ such that
$x_i \notin \{t_1,\ldots,t_{j-2}\}$.
Let $i$ be the largest such index.
Similarly, there must be at least one index
$k \in \{j+1,\ldots,n\}$ such that
$x_k \notin \{t_{j+1},\ldots,t_{n-1}\}$.
Let $k$ be the smallest such index.
We now have
$$
t_r = \begin{cases}
	x_{r+1}, & r=i,\ldots,j-1, \\
	x_r, & r=j,\ldots,k-1.
\end{cases}
$$
Since $t_j = x_j$ and $x_i < x_j < x_k$, we can
express $\bft$ as
$$ \bft = (1-\lambda) \bfa + \lambda \bfb, $$
where
$$ \lambda := \frac{x_j-x_i}{x_k-x_i}, $$
and
$$ \bfa := (t_1,\ldots,t_{j-1},x_i,t_{j+1},\ldots,t_{n-1}), $$
$$ \bfb := (t_1,\ldots,t_{j-1},x_k,t_{j+1},\ldots,t_{n-1}), $$
and by the convexity of $f$ in $[x_1,x_n]^{n-1}$,
$$ f(\bft) \le \max\{f(\bfa),f(\bfb)\}. $$
However,
$$ \bfa =
 (t_1,\ldots,t_{i-1},x_{i+1},\ldots,x_j,x_i,t_{j+1},\ldots,t_{n-1}), $$
and its elements are not non-decreasing, and so
$\bfa$ is not a vertex of $R(\bfx)$. But
we can re-order the elements to be non-decreasing, defining
$$ \bfa' := (t_1,\ldots,t_{i-1},
x_i,x_{i+1},\ldots,x_j,t_{j+1},\ldots,t_{n-1}). $$
Then, by the choice of $i$ and $j$, $\bfa'$ is a vertex of $R(\bfx)$
and $m(\bfa') = m(\bft)-1$, and
$f(\bfa') = f(\bfa)$ by the symmetry of $f$.

The point $\bfb$ may or may not be a vertex of $R(\bfx)$,
and there are two cases.
If $k=j+1$ then $\bfb$ is a vertex with $m(\bfb) = m(\bft) - 1$.
Otherwise, $k\ge j+2$ and then
$$ \bfb =
(t_1,\ldots,t_{j-1},x_k,x_{j+1},\ldots,x_{k-1},t_k,\ldots,t_{n-1}), $$
which is not a vertex.
Reordering the elements gives
$$ \bfb' := (t_1,\ldots,t_{j-1},
x_{j+1},\ldots,x_{k-1},x_k,t_k,\ldots,t_{n-1}), $$
and then $\bfb'$ is a vertex with $m(\bfb') = m(\bft)-1$
and $f(\bfb') = f(\bfb)$.
Thus we have proved (\ref{eq:reduce}).
\end{proof}

\section{The upper bound of Theorem~\ref{thm:main}}

For the upper bound of Theorem~\ref{thm:main},
let $n \ge 3$. For any $k \in \{1,\ldots,n\}$,
we can apply Lemma~\ref{lem:reductionupper}, and by
the induction hypothesis applied to $B_k(\bft)$,
and recalling (\ref{eq:pkVu}),
$$
\det(A) \le \frac{1}{c_{n-1}} e^{s(\bfx)y_k} V(\bfy)
\int_{R(\bfx)} V(\bft) \per(B_k(\bft)) \, d\bft.
$$
Taking the average of this over $k=1,\ldots,n$ implies
$$ \det(A) \le \frac{1}{n c_{n-1}} V(\bfy)
\int_{R(\bfx)} V(\bft) f(\bft) \, d\bft, $$
where
$$ f(\bft) := \sum_{k=1}^n e^{s(\bfx)y_k} \per(B_k(\bft)). $$
Then
$$ \det(A) \le \frac{1}{n c_{n-1}} 
\int_{R(\bfx)} V(\bft) \, d\bft \,
\max_{\bft \in R(\bfx)} f(\bft), $$
and by (\ref{eq:floater}),
$$ \det(A) \le \frac{1}{c_n} V(\bfx) V(\bfy)
\max_{\bft \in R(\bfx)} f(\bft). $$
Next, observe that swapping rows of $B_k(\bft)$ does not change its
permanent, and so $\per(B_k(\bft))$
is symmetric in $t_1,\ldots,t_{n-1}$, and therefore so is $f(\bft)$.
Further, for any real vector
$\bflambda = [\lambda_1,\ldots,\lambda_{n-1}]^T$,
the function
$$ g(\bft) := e^{\lambda_1 t_1 + \lambda_2 t_2 + \cdots + \lambda_{n-1} t_{n-1}}, $$
is convex because
the Hessian of $g$ is $H(g) = g \bflambda \bflambda^T$,
and so for any $\bfv \in \RR^{n-1}$,
$$ \bfv^T H(g) \bfv = (\bfv^T\bflambda)^2 g \ge 0. $$
Therefore, $\per(B_k(\bft))$, being the sum of
functions $g(\bft)$, is also convex, and so $f(\bft)$
is convex as well.
Therefore, by Lemma~\ref{lem:axial2},
$f$ attains its maximum in $R(\bfx)$ at one of the vertices
$\bfv_i$ in (\ref{eq:vi}).
Now observe that
$$
B_k(\bfv_i) = \left[
\begin{matrix}
e^{x_1 u_1} & \cdots & e^{x_1 u_{n-1}} \\
\vdots & & \vdots \\
e^{x_{i-1} u_1} & \cdots & e^{x_{i-1} u_{n-1}} \\
e^{x_{i+1} u_1} & \cdots & e^{x_{i+1} u_{n-1}} \\
\vdots & & \vdots \\
e^{x_{n-1} u_1} & \cdots & e^{x_{n-1} u_{n-1}}
\end{matrix}
\right].
$$
Multiplying its rows by the factors
$$ e^{x_1y_k},\ldots,e^{x_{i-1}y_k},
e^{x_{i+1}y_k},\ldots,e^{x_ny_k}, $$
respectively, we deduce that
$$ e^{s(\bfx)y_k} \per(B_k(\bfv_i)) = e^{x_iy_k} \per(A_{ik}), $$
where
$A_{ik}$ is the submatrix of $A$ formed by deleting
row $i$ and column $k$. Therefore, by the Laplace
expansion of $\per(A)$ by its $i$-th row we find
$$ f(\bfv_i) = \sum_{k=1}^n e^{x_iy_k} \per(A_{ik}) = \per(A),
\quad \hbox{for all $i=1,\ldots,n$.} $$
Thus the maximum of $f$ in $R(\bfx)$ is $\per(A)$,
which proves the upper bound of Theorem~\ref{thm:main}.

\section{Gaussian matrices}\label{sec:gaussian}

A matrix related to $A$ is the univariate Gaussian matrix
\begin{equation}\label{eq:Bgauss}
 B := [e^{-\lambda(t_j-t_i)^2/2}]_{i,j=1,\ldots,n},
\end{equation}
where $t_1,\ldots,t_n \in \RR$ are distinct and
$\lambda > 0$ is a parameter. This is the matrix that
arises when we interpolate
data at $t_1,\ldots,t_n$ with a linear combination of
the Gaussian radial basis functions
\begin{equation}\label{eq:gaussians}
\phi_j(t) := e^{-\lambda(t-t_j)^2/2}, \quad j=1,\ldots,n,
\end{equation}
see Micchelli~\cite{Micchelli:1986}, Fasshauer~\cite{Fasshauer:2011}.
We obtain bounds on $\det(B)$
from the bounds of~Theorem~\ref{thm:main}.
Let
$$ \overline{t} := \frac{1}{n} \sum_{i=1}^n t_i, \qquad
S := \sum_{i=1}^n (t_i - \overline{t})^2, \qquad
N := \binom{n}{2}. $$
\begin{corollary}\label{cor:gauss}
For $B$ in (\ref{eq:Bgauss}),
\begin{equation}\label{eq:gauss}
\frac{e^{-\lambda S}}{c_{n-1}} \le
\frac{\det(B)}{\lambda^N V(\bft)^2}
\le \frac{1}{c_n} \per(B).
\end{equation}
\end{corollary}

\begin{proof}
We may assume that $t_1 < \cdots < t_n$.
We multiply both row $i$ and column $i$ of $B$ by $e^{\lambda t_i^2/2}$,
and then
$$ \det(B) = e^{-\lambda(t_1^2 + \cdots + t_n^2)} \det(B'), $$
where
$$ B' := [e^{\lambda t_it_j}]_{ij=1,\ldots,n}. $$
With the substitutions $x_i := \sqrt{\lambda}t_i$ and
$y_j := \sqrt{\lambda}t_j$, the matrix $B'$ becomes $A$ in (\ref{eq:A}).
So the upper bound of Theorem~\ref{thm:main} implies
	$$ \det(B') \le \frac{1}{c_n} \lambda^N V(\bft)^2 \per(B'), $$
and thus the upper bound for $\det(B)$ follows.

The lower bound of Theorem~\ref{thm:main} implies
$$ \det(B') \ge \frac{1}{c_{n-1}} \lambda^N V(\bft)^2
e^{\lambda s(\bft)^2/n}, $$
and so
$$ \det(B) \ge \frac{1}{c_{n-1}} \lambda^N V(\bft)^2
e^{-\lambda ((t_1^2 + \cdots + t_n^2)- s(\bft)^2/n)}.
$$
As is well known,
$$ \sum_{i=1}^n t_i^2 - \frac{1}{n} \Big(\sum_{i=1}^n t_i \Big)^2 = S $$
(by replacing $t_i$ by $(t_i-\tbar) + \tbar$ in the first sum)
and thus we obtain the lower bound of (\ref{eq:gauss}).
\end{proof}

Various methods have been proposed for optimizing
the shape parameter in radial basis interpolation;
see~\cite{Rippa:1999}, \cite[Section 5.2]{Fasshauer:2011},
\cite{LarssonFornberg:2005}.
As a general rule, for the Gaussians (\ref{eq:gaussians}),
it is usual to choose $\lambda$ at least to be inversely
proportional to the squares of the distances among the $t_i$.
We can use the lower bound of Corollary~\ref{cor:gauss} to
make a choice of $\lambda$ that agrees with that general rule.
We can write the lower bound as
$$ \det(B) \ge \frac{1}{c_{n-1}} V(\bft)^2 F(\lambda), $$
where
$$ F(\lambda) := \lambda^N e^{-\lambda S}. $$
The function $F(\lambda)$
approaches zero both as $\lambda \to 0$ and as $\lambda \to \infty$,
and this suggests choosing the parameter~$\lambda$
to maximize $F$. The maximum is attained when $F'(\lambda) = 0$,
i.e., when
$$ \lambda = \lambda_* := \frac{N}{S} = 
  \frac{n(n-1)}{2} \Big/ \sum_{i=1}^n (t_i - \overline{t})^2. $$

The following numerical examples were used to test various
values of $\lambda$.
We took $n=101$ and $t_i = (i-1)/100$, $i=1,2,\ldots,101$, and
$$ f_1(x) := x \cos(3\pi x), \qquad f_2(x) := 2 x\sqrt{1-x},
\qquad 0 \le x \le 1. $$
We computed the interpolant for $\lambda = 10^k \lambda_*$,
with $k=-2,-1,0,1,2$.
The max errors in $[0,1]$ for $f_1$ and $f_2$ are shown
in Table~\ref{table:1},
and we see that $k=0$ gives the smallest error for both $f_1$ and $f_2$,
suggesting that $\lambda_*$ seems to offer a good choice
of shape parameter.
\begin{table}\centering
\begin{tabular}{c|c|c}
$\lambda$ \hfill & Error, $f_1$ & Error, $f_2$ \\
\hline
$0.01 \times \lambda_*$ \hfill & $2.8 \times 10^{-6}$ & $0.47$ \\
$0.1 \times \lambda_*$ \hfill & $5.8 \times 10^{-8}$ & $0.030$ \\
$\lambda_*$ \hfill & $5.4 \times 10^{-8}$ & $0.020$ \\
$10 \times \lambda_*$ \hfill & $0.013$ & $0.044$ \\
$100 \times \lambda_*$ \hfill & $0.13$ & $0.10$ \\
\end{tabular}
\caption[]{Effect of choice of $\lambda$} \label{table:1}
\end{table}
Figure~\ref{fig:f1} shows the Gauss interpolant to $f_1$
with $\lambda=\lambda_*$.
\begin{figure}[ht]
\centerline{\includegraphics[width=0.5\textwidth]{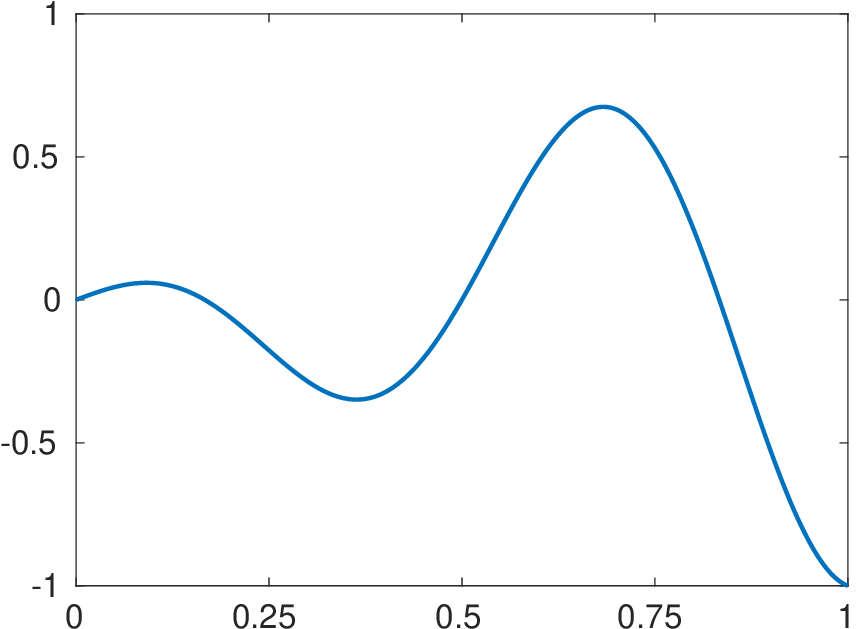}}
\caption{Gauss interpolant to $f_1$ with $\lambda = \lambda_*$.} 
\label{fig:f1}
\end{figure}

\Acknowledgement{I wish to thank the referees for their suggestions
for improvements, especially the importance of the unitary matrix integral,
to Stefano de Marchi and Jonathan Novak for help with references,
to Geir Dahl for the reminder about the elegant Birkhoff-von Neumann theorem,
and to Espen Sande for the encouragement to persevere with this project.}

\bibliography{exp_short}

@Article{Ait-Haddou:2019,
  author =    {R. Ait-Haddou and M.-L. Mazure},
  title =     {The fundamental blossoming inequality in {C}hebyshev spaces
	       --- {I}: applications to {S}chur functions},
  journal =   {Found. Comp. Math.},
  year =      {2018},
  volume =    {18},
  pages =     {135-158}
}

@Article{Alfeld:2005,
  author =    {P. Alfeld and L. L. Schumaker},
  title =     {A ${C}^2$ trivariate macro-element based on the {C}lough-{T}ocher-split of a tetrahedron},
  journal =   {{C}omp. {A}ided {G}eom. {D}esign.},
  year =      {2005},
  volume =    {22},
  pages =     {710-721}
}

@Book{Bapat:1997,
  author =    {R. B. Bapat and T. E. S. Raghavan},
  title =     {Nonnegative matrices and applications},
  publisher = {Cambridge University Press},
  year =      {1997}
}

@Article{deBoor:2005,
  author =    {C. {de Boor}},
  title =     {Divided differences},
  journal =   {Surveys in Approx. Theory},
  year =      {2005},
  volume =    {1},
  pages =     {46-69}
}

@Article{BosdeMarchi:2011,
  author =    {L. Bos and S. de Marchi},
  title =     {On optimal points for interpolation},
  journal =   {Dolomites {R}esearch {N}otes on {A}pproximation},
  year =      {2011},
  volume =    {4},
  pages =     {8-12}
}

@Article{CoopeRenaud:2009,
  author =    {I. D. Coope and P. F. Renaud},
  title =     {Trace inequalities with applications to orthogonal regression
               and matrix nearness problems},
  journal =   {J. Inequalities Pure Appl. Math.},
  year =      {2009},
  volume =    {10},
  pages =     {Paper No. 92, pp. 1-7} 
}

@Book{FallatJohnson:2011,
  author =    {S. M. Fallat and C. R. Johnson},
  title =     {Totally non-negative matrices},
  publisher = {Princeton University Press},
  year =      {2011}
}

@Article{Fasshauer:2011,
  author =    {G. E. Fasshauer},
  title =     {Positive definite kernels: past, present and future},
  journal =   {Dolomites {R}esearch {N}otes on {A}pproximation},
  year =      {2011},
  volume =    {4},
  pages =     {21-63}
}

@Article{Floater:2003,
  author =    {M. S. Floater},
  title =     {Error formulas for divided difference
     expansions and numerical differentiation},
  journal =   {J. Approx. Theory},
  year =      {2003},
  volume =    {122},
  pages =     {1-9} 
}

@Article{Floater:2023,
  author =    {M. S. Floater},
  title =     {On a conjecture concerning interpolation by
	       bivariate {B}ernstein polynomials},
  journal =   {J. Approx. Theory},
  year =      {2023},
  volume =    {293},
  pages =     {105920, pp. 1-14} 
}

@Misc{Floater:2026a,
  author =    {M. S. Floater},
  title =     {A divided difference identity for a class of multiple integrals},
  year =      {2026},
  howpublished =  {preprint}
}

@Article{Gantmakher:1937,
  author =    {F. R. Gantmakher and M. G. Krein},
  title =     {Sur les matrices compl{\`e}tement non n{\'e}gatives
	       et oscillatoires},
  journal =   {Compositio Math.},
  year =      {1937},
  volume =    {4},
  pages =     {445-476} 
}

@Article{GrossRichards:1989,
  author =    {K. Gross and D. St. P. Richards},
  title =     {Total positivity, spherical series, and hypergeometric
               functions of matrix argument},
  journal =   {J. Approx. Theory},
  year =      {1989},
  volume =    {59},
  pages =     {224-246} 
}

@Book{Hardy:1952,
  author =    {G. H. Hardy and J. E. Littlewood and G. Polya},
  title =     {Inequalities, 2nd edition},
  publisher = {Cambridge University Press},
  year =      {1952}
}

@Article{HarishChandra:1957,
  author =    {Harish-Chandra},
  title =     {Differential operators on a semisimple {L}ie algebra},
  journal =   {Amer. J. Math.},
  year =      {1957},
  volume =    {79},
  pages =     {87-120} 
}

@Book{Isaacson:1994,
  author =    {E. Isaacson and H. B. Keller},
  title =     {Analysis of numerical methods},
  publisher = {Dover},
  address =   {New York},
  year =      {1994}
}

@Article{ItzyksonZuber:1980,
  author =    {C. Itzykson and J. B. Zuber},
  title =     {The planar approximation. {II}},
  journal =   {J. Math. Phys.},
  year =      {1980},
  volume =    {21},
  pages =     {411-421}
}

@Article{Jaklic:2014,
  author =    {G. Jakli\v{c} and T. Kandu\v{c}},
  title =     {On positivity of principal minors of bivariate {B}ezier collocation matrix},
  journal =   {J. Appl. Math. and Comp.},
  year =      {2014},
  volume =    {227},
  pages =     {320-328}
}

@Article{LarssonFornberg:2005,
  author =    {E. Larsson and B. Fornberg},
  title =     {Theoretical and computational aspects of multivariate
	       interpolation with increasingly flat radial basis functions},
  journal =   {Comp. Math. Appl.},
  year =      {2005},
  volume =    {49},
  pages =     {103-130}
}

@Book{Macdonald:1995,
  author =    {I. G. Macdonald},
  title =     {Symmetric functions and {H}all polynomials},
  publisher = {Oxford University Press},
  year =      {1995}
}

@Article{deMarchi:2001,
  author =    {S. de Marchi},
  title =     {Polynomials arising in factoring generalized {V}andermonde
	       determinants: an algorithm for computing their coefficients},
  journal =   {Math. Comp. Modelling},
  year =      {2001},
  volume =    {34},
  pages =     {271-281}
}

@Article{Micchelli:1986,
  author =    {C. A. Micchelli},
  title =     {Interpolation of scattered data: distance matrices and conditionally positive definite functions},
  journal =   {Const. Approx.},
  year =      {1986},
  volume =    {2},
  pages =     {11-22}
}

@Book{Pinkus:2009,
  author =    {A. Pinkus},
  title =     {Totally positive matrices},
  publisher = {Cambridge University Press},
  year =      {2009}
}

@Book{Polya:1976,
  author =    {G. P\'{o}lya and G. Szeg\H{o}},
  title =     {Problems and theorems in analysis {II}},
  publisher = {Springer-Verlag},
  address =   {New York},
  year =      {1976}
}

@Article{Rippa:1999,
  author =    {S. Rippa},
  title =     {An algorithm for selecting a good value for the parameter $c$
	       in radial basis function interpolation},
  journal =   {Adv. Comp. Math.},
  year =      {1999},
  volume =    {11},
  pages =     {193-210}
}

@Book{Steffensen:1927,
  author =    {J. F. Steffensen},
  title =     {Interpolation},
  publisher = {Williams and Wilkins},
  address =   {Baltimore},
  year =      {1927}
}

@Article{Yarotsky:2013,
  author =    {D. Yarotsky},
  title =     {Univariate interpolation by exponential functions and
               {G}aussian {RBF}s for generic sets of nodes},
  journal =   {J. Approx. Theory},
  year =      {2013},
  volume =    {166},
  pages =     {163-175} 
}

\end{document}